%% file: ex_article.tex
\documentclass[hidelinks,onefignum,onetabnum]{siamart250211}

\input{ex_shared}

\ifpdf
\hypersetup{
  pdftitle={Regular Solutions to the Beckmann Optimal Transport},
  pdfauthor={Hanno Gottschalk \and Tobias J.\ Riedlinger}
}
\fi

\begin{document}
%\author{Hanno Gottschalk \thanks{Institute of Mathematics, technical University Berlin, Berlin 10623, Germany 
%  (\email{\{gottschalk,riedlinger\}@math.tu-berlin.de}).}\and Tobias Riedlinger ${}^\dagger$}

\maketitle

% REQUIRED
\begin{abstract}
Beckmann's problem in optimal transport minimizes the total squared flux in a continuous transport problem from a source to a target distribution. In this article, the regularity theory for solutions to Beckmann's problem in optimal transport is developed utilizing an unconstrained Lagrangian formulation and solving the variational first order optimality conditions. It turns out that the Lagrangian multiplier that enforces Beckmann's divergence constraint fulfills a Poisson equation and the flux vector field is obtained as the potential's gradient. Utilizing Schauder estimates from elliptic regularity theory, the exact Hölder regularity of the potential, the flux and the flow generating is derived on the basis of Hölder regularity of source and target densities on a bounded, regular domain. If the target distribution depends on parameters, as is the case in conditional (``promptable'') generative learning, we provide sufficient conditions for separate and joint Hölder continuity of the resulting vector field in the parameter and the data dimension. Following a recent result by Belomnestny \emph{et al.}, one can thus approximate such vector fields with deep \texttt{ReQu} neural networks in $C^{k,\alpha}$-Hölder norm. We also show that this approach generalizes to other probability paths, like Fisher-Rao gradient flows.
\end{abstract}

% REQUIRED

% REQUIRED

\section{Introduction}
Optimal transport has become a source of inspiration for generative learning. The existence of a transport map that shifts a source distribution -- oftentimes some kind of noise -- to an involved target distribution is a precondition to prove the consistency of generative learning schemes. Such proofs often require the approximation of transport maps with deep neural networks. While universal approximation holds for continuous or Lebsgue integrable maps  \cite{cybenko1989approximation}, quantitative approximation results are required for deriving approximation rates quantitatively, which enters proofs of consistency with explicit rate estimates \cite{belomestny2023simultaneous,puchkin2024rates}. This underpins the importance of regularity results in optimal transport, like e.g. Caffarelli's theorem on the Hölder regularity of the Brenier potential \cite{brenier1991polar,caffarelli1996boundary,villani2008optimal}.

Recent generative learning has focused on flow based transport maps \cite{chen2018neural,lipmanflow,marzouk2024distribution,ehrhardt2025numericalstatisticalanalysisneuralode}, which deform the source distribution continuously into a target distribution via a probability path of measures depending on an auxiliary time parameter\cite{lipmanflow,wald2025flow}, with the linear interpolation between densities as the easiest of such paths. The flow along a path is then achieved by pushing the source measure forward via the flow map associated with an ordinary differential equation. In learning algorithms, the vector fields are then learned by deep neural networks \cite{chen2018neural,lipmanflow}. Again, the regularity of the vector fields determines the expressivity of the neural networks required \cite{marzouk2024distribution}. As recent results provide rates on the approximation of Hölder functions with neural networks with $\mathtt{ReQU}=\mathtt{ReLU}^2$-activation (or more generally $\mathtt{ReLU}^s$ activation) \cite{belomestny2023simultaneous,drygala2025learning}, precise results on the regularity of transport vector fields are of interest, see e.g.  for a recent result based on the Rosenblatt-Knothe map of the probability path \cite{wang2022minimax}.

The vector fields that induce a flow along a prescribed probability path are not unique. If fact, by Liouville's formula, at any time the vector field can be modified by adding a divergence free vector field. In order to avoid such spurious degrees of freedom, optimal transport often imposes additional optimality conditions on the flow like, e.g., minimal $L^p$ norms of the velocity vector fields \cite{COTFNT,SantambrogioFilippo2015OTfA,villani2008optimal}, with $p=2$ a frequent choice. 

The Beckmann problem of continuous optimal transport can be seen as a variant of this strategy. Here usually the simple linear interpolation path is followed. Instead of imposing an optimality condition to minimize the $L^2$-norm of the velocity vector field \cite{COTFNT} or the energy \cite{benamou2000computational}, the optimality in Beckmann's problem concerns the $L^2$-norm of the flux vector field, i.e. the product of the flow generating vector field at time $t$ with the density of the probability path at the same time\cite{SantambrogioFilippo2015OTfA}. What makes this approach attractive is the fact, that one can achieve the linear interpolation flow path with a flux field that is constant in time.

The theory of existence of solutions to Beckmann's problem is well settled \cite{SantambrogioFilippo2015OTfA}. In contrast, results on the regularity of Beckmann's flux fields and associated flow generating vector fields seem to be sparse,  see however \cite{Lorenz2022} for recent results on $L^p$-regularity and \cite{dweik2023optimal} for $W^{1,p}$-regular solutions for transport cost proportional to the distance and probability paths in $L^p$.

In this work, for the case of positive Hölder densities for source and target distributions on a bounded, regular domain, we give optimal results on the Hölder regularity for the Beckmann flux field, and the associated vector field and the flow map using Schauder estimates from elliptic regularity theory \cite{agmon1959estimates,gilbarg1977elliptic}. Elliptic partial differential equations come into play by writing the constrained Beckmann problem into a unconstrained form via the Lagrangian approach and solving the variational equation for the Lagrangian multiplier to Beckmann's divergence constraint. A similar strategy has been applied by Bennau and Brenier \cite{benamou2000computational} for  flows minimizing the potential energy. 

In our case, this implies that the Lagrangian multiplier solves Poisson's equation with zero Neumann boundary conditions and the right hand side equal to the difference of the source and the target density or, in the case of non trivial probability paths, the time derivative of the density of the probability path. The Beckmann flux field then is obtained as the gradient to the Lagrangian multiplier acting as a 'potential'. Existence of the Solution to the Poisson equation follows from Fredholm's alternative \cite{nardi2015schauder} and the optimal regularity results can then be deduced from the classical results of Agmon, Douglas and Nirenberg \cite{agmon1959estimates}. Regularity of vector fields now can be deduced by division of the flux by the density and regularity of the flow maps is derived by numerical integration \cite{effland2020convergence}. In this way, we obtain optimal regularity results for flow fields which are fulfilling the requirements of mathematical studies in generative learning.

In the following we extend our considerations to parametric optimal transport, which relates to conditional generative learning where the target distribution may depend on a given ``prompt''. 
By studying parametric Beckmann problems we derive Hölder continuity in the parameters of the target distribution for the Beckmann vector field, provided the density of the target distribution is sufficiently jointly regular in parameter and data space.
To this end we exploit Schauder estimates and formulate a sufficient requirement that the parametric source and target distributions fulfill an abstract Hölder condition as a Banach-valued function.
Different notions of Hölder continuity are derived and sufficient conditions are identified such that the resulting parametric Beckmann vector field can be uniformly approximated by deep \texttt{ReQU} neural networks in Hölder norms. 
We thereby close a gap in the mathematical understanding of conditional generative learning.
We also show that our construction generalizes to arbitrary regular probability paths with densities lower bounded away from zero, like the certain Fisher-Rao gradient flows.

Our paper is organized as follows: In Section \ref{sec:liouville-formula} we recall the basics of probability densities transforming under the flow induced by a vector field. In Section \ref{sec:Beckmann} we introduce the Beckmann problem and transform it into the unconstrained form in Section \ref{sec:Unconstrained}. Here we also derive our potential equation of Poisson type. Section \ref{sec:Schauder} discusses the solution of the potential equation via the Fredholm alternative and the Schauder estimates for the potential, from which we derive the regularity results for the flux and flow vector fields and the flow maps. 
Section \ref{sec:Paths} generalizes from the interpolating probability path to general paths. 
In Section \ref{sec:Conclusion} we give our conclusions and a short outlook on future research.

\paragraph{Notation}
Throughout this paper, we will denote the vector of spatial partial derivatives of a function $f:\R^d \to \R$, $x \mapsto f(x)$ by $\bm\nabla f$.
For $f$ sufficiently smooth and a multi-index $\bm\beta = (\beta_0, \ldots, \beta_d) \in \N_0^d$ with degree $|\bm\beta| = \sum_{j = 1}^d \beta_j$, we define the $\bm\beta$-derivative of $f$ by
\begin{equation}
    \label{eq:definition-multiindex-derivative}
    \nabla^{\bm\beta} f(x) = \frac{\partial^{|\bm\beta|}}{\partial x_1^{\beta_1}\cdots \partial x_d^{\beta_d}} f(x).
\end{equation}
Vector field functions $\bm{w}: \R^d \to \R^d$ are denoted with bold letters and we write $\nabla \bm{w}:\R^d \to \R^{d \times d}$ for the Jacobian matrix of $\bm{w}$ and $\bm\nabla \cdot \bm{w}: \R^d \to \R$ for the divergence field of $\bm{w}$.
We also define the Laplacian of a scalar function $f$ by $\Delta f = \bm{\nabla} \cdot \bm{\nabla}f = \sum_{j = 1}^d \tfrac{\partial^2}{\partial x_j^2} f$.
In \cref{sec:parametrized-beckmann}, we will deal with functions $f:\R^d \times \R^q \to \R$ and $\bm{w}: \R^d \times \R^q \to \R^d$ which depend on an additional parameter $\theta \in \R^q$.
We respectively denote the parameter gradient vector of $f$ by $\overline{\bm\nabla}f(x, \theta) \in \R^q$, and the multi-index partial derivative by $\overline{\nabla}^{\bm{m}} f(x, \theta) \in \R$, where $\bm{m} \in \N_0^q$ analogously to \cref{eq:definition-multiindex-derivative}.
In the context of joint differentiability on $\R^d\times \R^q$, we denote mixed partial derivatives with multi-index $\bm{m}\in \N_0^{d + q}$ by $\partial^{\bm{m}}$.
For $X$ a Banach space, $k \in \N$ and $f \in C^k(\R^q;X)$, we denote by $D^{\bm{m}} f$ the Fr\'{e}chet partial derivatives of $f$, also analogously to \cref{eq:definition-multiindex-derivative} whenever $|\bm{m}| \leq k$.

\section{Transport Maps and Flows}\label{sec:liouville-formula}
Most modern generative algorithms learn transport maps $\bm\Phi:\Omega\to\Omega$ which map a source distribution $\nu$ to a target distribution $\mu$.
That is, $\mu$, $\nu$ are probability measures on a common domain $\Omega\subseteq \R^d$ with Lipschitz boundary and we require $\bm\Phi_*\nu=\mu$ for $\bm\Phi$ to be a transport map, where for any Borel set $A$,  $\bm\Phi_*\nu(A)=\nu(\bm\Phi^{-1}(A))$ is the image measure of $\nu$ under $\bm\Phi$.
Optimal transport theory deals with the existence and properties of transport maps. 
Flow-based generative models view $\bm\Phi$ as time $t=1$ end points of ODE flows $\bm\Phi_{0, t}(\cdot)|_{t = 1}$ with respect to some velocity vector field $\bm \xi: [0,1]\times\Omega\to \R^d$, i.e.
\begin{equation*}
    \frac{\mathrm{d}}{\mathrm{d} t}\bm\Phi_{s,t}(x) = \bm\xi(t, \bm\Phi_{s,t}(x)), \quad\bm\Phi_{s,s}(x)=x \qquad \forall s \in [0,1].
\end{equation*}
In the following, we assume that $\bm \xi(t, \cdot)\cdot \bm \eta = 0$ holds on the boundary $\partial \Omega$, where $\bm\eta$ is the outward normal vector field of $\Omega$.
Under sufficient regularity of $\bm\xi(t)$, namely that the map $x \mapsto \bm\xi(t, x)$ is Lipschitz continuous for all $t \in [0,1]$, the flow $\bm\Phi_{0,t}(\bm\xi)$ exists and is unique for $t \in [0, 1]$ and $x \mapsto \bm\Phi_{0, t}(\bm\xi)(x)$ is a diffeomorphism for all $t \in [0,1]$.
The flow endpoint $\bm\Phi(\bm\xi) = \bm\Phi_{0,1}(\bm\xi)(\cdot):\Omega \to \Omega$ is then also a diffeomorphism.

For a generic bijective and differentiable map $\bm\Phi$, the density of $\bm\Phi_*\nu$, given that $\nu = \rho_\nu\cdot \dif x$ has a probability density $\rho_\nu$ with respect to the Lebesgue measure $\dif x$, the density of $\bm\Phi_*\nu $ is given by the transformation of densities formula \cite[Thm.\ IV.8.9]{wernerEinfuehrungHoehereAnalysis2009a}
\begin{equation}
    \label{eq:transformation_of_densities}
    \rho_{\bm\Phi} = [\rho_\nu\circ \bm\Phi^{-1}] \cdot \left|\det \nabla(\bm\Phi^{-1})\right|,
\end{equation}
$\det(A)$ denotes the determinant of a quadratic matrix $A$.
Considering the probability path $[\Phi_{0,t}(\bm\xi)]_*\nu$, its density $\rho(t) = \rho_\nu \circ \Phi_{0,t}(\bm\xi)^{-1} \cdot |\det \nabla\Phi_{0,t}(\bm \xi)^{-1}|$ satisfies the (weak) transport equation \cite[Thm.~4.4]{SantambrogioFilippo2015OTfA}
\begin{equation}
\label{eq:continuityEquation}
    \frac{\partial}{\partial t} \rho(t) + \bm\nabla \cdot (\rho(t) \cdot \bm\xi(t, \cdot)) = 0, \quad \text{a.e. in } \Omega.
\end{equation}
Coming back to the transport problem of mapping $\nu$ to $\mu$, where also the latter has a Lebesgue density $\rho_\mu$, any time-independent vector field $\bm w: \Omega \to \Omega$ giving rise to a transport map $\bm\Phi(\bm\xi)$ must necessarily satisfy the divergence property
\begin{equation}
    \label{eq:flux-condition}
    \bm\nabla \cdot  \bm w = \rho_\mu - \rho_\nu, \quad \text{a.e. in } \Omega
\end{equation}
which requires suitable regularity of $\bm w$.
We call a vector field that satisfies \cref{eq:flux-condition} a ``flux field''.
This opens the question when a flow exists that transforms a given source distribution $\nu = \rho_\nu\cdot \dif x$ to a target distribution $\mu = \rho_\mu \cdot \dif x$ and what can be said about the regularity the underlying vector field $\bm\xi$.

\section{Beckmann's problem} \label{sec:Beckmann}
Here we give an account of the above-mentioned regularity problem based on potential theory and Schauder estimates for elliptic partial differential equations (PDE)~\cite{agmon1959estimates,gilbarg1977elliptic}. 
Our solution is related to the so called Beckmann problem in optimal transport \cite[Chapter 4]{SantambrogioFilippo2015OTfA}.  

We are looking for a vector field $\bm w:\Omega \to \R^d$ which models the total in-flux or out-flux of a differential volume inside a bounded domain $\Omega\subseteq \R^d$ with Lipschitz boundary $\partial\Omega$.  
The problem we consider is to solve
\begin{equation}
\label{eq:min_sq_norm_w}
J(\bm w)=\frac{1}{2}\int_\Omega |\bm w|^2\, \mathrm{d}x \rightarrow \min ~, 
\end{equation}
where the minimization takes place under vector fields $\bm w$ which fulfill the constraint
\begin{equation}
\label{eq:divergenceConstraint}
\bm\nabla \cdot \bm w = f ~~\text{on}~~\Omega, ~~\text{and}~~\bm w\cdot \bm\eta=0,~~\text{on}~~\partial\Omega~~(a.e.),    
\end{equation}
where $f = \rho_\nu-\rho_\mu$ and $\bm\eta$ is the outward pointing normal vector field on $\partial\Omega$.
Here, $\bm w$ may be in some Sobolev space $W^{1,p}(\Omega,\R^d)$, $p\in[2,\infty)$, that allows taking a derivative and a trace on $\partial\Omega$ \cite{adams2003sobolev}. 
Note that the classical Beckmann problem replaces $|\bm w|^2$ with $|\bm w|$, but here we exclusively deal with the case of quadratic cost in the flux. 

The existence of optimal solutions $\bm w$ can be proven under suitable conditions but the regularity properties of $\bm w$ may not yet be strong enough to assure the existence of flows connected to $\bm\xi$.
However, the regularity of $\bm w$ is known to experts in the field~\cite[Chapter 4]{SantambrogioFilippo2015OTfA} but to our knowledge nowhere explicitly present in the contemporary literature.
See however~\cite{dweik2022w1,Lorenz2022} for some recent progress in a more irregular setting that we consider here. 
In the following we identify a sufficiently strong set of assumptions that allows us to derive Hölder regularity for the vector field $\bm w$.

Let us shortly outline the differences to the approach of Benamou and Brenier \cite{benamou2000computational}. Here the optimization problem minimizing the kinetic energy of the flow
\[
\frac{1}{2}\int_0^1\int_\Omega |\bm\xi(t, x)|^2 \cdot  \rho(t, x)\, \mathrm{d}x\,\mathrm{d}t\to\min 
\]
is minimized in $\rho$ and $\bm\xi$ and \eqref{eq:continuityEquation} is treated as a constraint, see also \cite{COTFNT} where the modified problem $\frac{1}{2}\int_0^1\int_\Omega|\bm\xi(t, x)|^2\, \mathrm{d}x\,\mathrm{d}t\to\min 
$ is considered as a minimization in $\bm\xi$. In both cases, the non-linearity of the constraint equation \eqref{eq:continuityEquation} imposes analytical challenges, which here can be avoided.  

\subsection{Unconstrained formulation of the Beckmann problem} \label{sec:Unconstrained}

Similar as in \cite{benamou2000computational}, we pass from the constraint formulation of Beckmann's problem to an unrestricted formulation using Langrangian multipliers $u:\Omega\to\R$ and $v:\partial\Omega\to \R$. We thus consider the Lagrangian functional
\begin{equation}
    \label{eq:Lagrangian}
    \mathcal{L}(\bm w,u,v)=\frac{1}{2}\int_\Omega |\bm w|^2\, \mathrm{d}x+\int_\Omega (\bm\nabla \cdot \bm w - f)\,u\, \mathrm{d}x+\int_{\partial\Omega} (\bm\eta \cdot \bm w)\, v\, \mathrm{d}S,
\end{equation}
where $\dif S$ stands for the induced surface volume element on $\partial\Omega$. If $u$ and $v$ are chosen from sufficiently large linear function spaces, the constraint minimization of \eqref{eq:min_sq_norm_w} is equivalent to 
\begin{equation}
    \label{eq:unconstraint_min_w}
    \min_{\bm w} \sup_{u,v} \mathcal{L}(\bm w,u,v)
\end{equation}
In fact, should one of the constraints in \eqref{eq:divergenceConstraint} be violated, one could find suitable Lagrange multipliers $u$, $v$ such that the last two terms in \eqref{eq:Lagrangian} diverge. 
Thus the supremum of these two terms in $u$ and $v$ is infinite, which implies that such a $\bm w$ can not be a minimizer.

We now give a rigorous mathematical formulation of \eqref{eq:unconstraint_min_w} based on the Hölder regularity of the densities of source and target measures.
To this purpose, let \linebreak $C^{k,\alpha}(\Omb,\R^q)$ be the set of $k$ times differentiable functions from $\bar \Omega$ to $\R^q$ with $\alpha$-H\"older continuous $k$-th derivative, $k\in\N_0$, $\alpha \in (0,1]$, $q\in\N$. We utilize the following convention on the Hölder norms
\begin{align*}
 \| f\|_{C^{k,\alpha}}&=\max \left\{\|\nabla^{\bm\beta} f_j\|_{C(\Omb,\R)}:j=1,\ldots,q,\, |\bm\beta|\leq k\right\} \\
 &+\max \left\{\sup_{x \neq x'\in\bar \Omega}\frac{\left|\nabla^{\bm\beta} (f_j(x)-f_j(x'))\right|}{|x-x'|^\alpha}:j=1,\ldots,q,\, |\bm\beta|=k\right\}.
\end{align*}
If $\Omega$ has a $C^{k,\alpha}$ boundary, i.e. the boundary can be straightened by $C^{k,\alpha}$ hemisphere transformations, then we write $g\in C^{k,\alpha}(\partial\Omega)$ if $g:\partial\Omega\to \R$ such that for any of the hemisphere transformations, the mapping from the boundary of the hemisphere to $\partial \Omega$ composed with the hemisphere transformation restricted to the straightened boundary is locally given by a function from $C(\R^{d-1},\R^d)$, see~\cite{agmon1959estimates,gilbarg1977elliptic} for further details.  

\begin{assumptions}
Let $k\in \N_0$ and $\alpha\in(0,1)$.
    \begin{enumerate}[label=A\arabic*)]
        \item Let $\Omega \subseteq \R^d$ be bounded with $C^{k+2,\alpha}$ boundaries.
        \label{assumption:smooth-domain}
        \item Let $\rho_\mu>0$ and $\rho_\nu>0$ on $\Omb$ such that 
        \begin{equation}
            \label{eq:kappaPositive}
            \kappa:=\min\left\{\inf_{x\in \Omb} \rho_\mu(x),\inf_{x\in \Omb} \rho_\nu(x)\right\}>0.
        \end{equation}
        \label{assumption:uniform-lower-bound}
        \item Let $\rho_\nu, \rho_\mu \in C^{k,\alpha}(\Omb,\R)$ for $\alpha \in (0, 1]$.
        \label{assumption:hoelder-differentiable-densities}
    \end{enumerate}
\end{assumptions}
Also, we define the finite joint upper bound $K := \max \{ \sup_{x \in \Omb} \rho_\mu(x), \sup_{x \in \Omb} \rho_\nu(x)\}$ which exists by continuity.

Let us now specify suitable spaces for the optimization variables, i.e. we chose $\bm w\in W^{1,p}(\Omega,\R^d)$, $p\geq 2$, $u\in C^{k+2,\alpha}(\Omb)$ and $v\in C^0(\partial\Omega)$. 
By the Sobolev embedding theorem \cite{adams2003sobolev}, $\bm w$ has a continuous extension to $\Omb$ if $p\geq d$ and admits a trace with values in $L^q(\partial\Omega)$ for $p<d$ if $1\leq q\leq \frac{(d-1)p}{(d-p)}$ if $1\leq p<d$. Furthermore, under the given conditions the integration by parts formula holds for the first derivatives of $\bm w$ multiplied by $u$ \cite{adams2003sobolev}. 

We investigate the first order optimality conditions of  problem \eqref{eq:unconstraint_min_w}
\begin{subequations}
\begin{align}
    \label{eq:first_order_opt_a}
    0&= \frac{\delta}{\delta v}\mathcal{L}(\bm w,u,v)=\bm \eta\cdot \bm w~~\text{on}~~\partial\Omega ~~(a.e.)\\
    \label{eq:first_order_opt_b}
    0&= \frac{\delta}{\delta u}\mathcal{L}(\bm w,u,v)=\bm\nabla\cdot \bm w - f =0~~\text{on}~~\Omega ~~(a.e.)\\
    \label{eq:first_order_opt_c}
    0&= \frac{\delta}{\delta \bm w}\mathcal{L}(\bm w,u,v)=\bm w - \bm \nabla u=0~~\text{on}~~\Omega ~~(a.e.)
\end{align}
\end{subequations}
To derive \eqref{eq:first_order_opt_c}, we used integration by parts by the divergence theorem and \eqref{eq:first_order_opt_a}. 
Since $p\geq 2$, all terms can be interpreted as Fr\'echet derivatives of $\mathcal{L}$ in the respective spaces.  

Combining  \eqref{eq:first_order_opt_a}-- \eqref{eq:first_order_opt_c}, we see that the Lagrangian function $u:\Omega\to\R$ has to solve the Poisson equation with Neumann boundary conditions
\begin{equation}
\label{eq:Poisson}
\Delta u = f \quad \text{on } \Omega, \qquad \bm \nabla u\cdot \bm \eta=0 \quad \text{ on } \partial\Omega,
\end{equation}
as well as $\bm w = \bm \nabla u$.

The occurrence of a potential function $u$ is typical. 
E.g. in the Benamou-Brenier approach \cite{benamou2000computational}, the velocity field $\bm \xi$ is obtained as a gradient of a potential $\bm \xi (t, x) = \bm \nabla u (t, x)$, where $u(t, x)$ fulfills the Hamilton-Jacobi equation $\frac{\partial}{\partial t} u(t, x) + \frac{1}{2}|\bm \nabla u(t, x)|^2=0$. As this equation is non-linear, the regularity theory of the Hamilton Jacobi equation is more involved than in our case \cite{crandall1983viscosity}, in which we only have to deal with an elliptic equation.  

\subsection{Potential theory and Schauder estimates} \label{sec:Schauder}

In this section we derive the existence and regularity of the solution $u$ of the first order optimality conditions of the Beckmann problem in the form \eqref{eq:Poisson}  using elliptic regularity theory~\cite{agmon1959estimates,gilbarg1977elliptic}.
Let us first recall a result on the existence of strong solutions of the Poisson equation with Neumann boundary conditions via the Fredholm alternative:

\begin{theorem}[Existence and $C^{2,\alpha}$ Schauder estimate~\cite{nardi2015schauder}]
\label{theorem:Schauder}
Let $d>2$, $\alpha\in(0,1)$ and $\Omega$ fulfill assumption \ref{assumption:smooth-domain} for $k=0$. 
Let $f\in C^{0,\alpha}(\Omega,\R)$ be given and $g:\partial \Omega\to \R$ be in $C^{1,\alpha}(\partial\Omega,\R)$. 
If $\int_\Omega f\, \mathrm{d}x+\int_{\partial\Omega} g \,\mathrm{d} S=0$,
then there exits a solution $u\in C^{2,\alpha}(\overline{\Omega},\R)$ to the Poisson equation 
with Neumann boundary conditions \cref{eq:Poisson} 
which is unique up to a constant.
Furthermore, there exists a constant $C=C_1(\Omega, d, \alpha)$ such that
\begin{equation}
\label{eq:2ndOrderSchauder}
\left\| u-\frac{1}{|\Omega|}\int_\Omega u\,\mathrm{d}x\right\|_{C^{2,\alpha}(\Omega,\R)}\leq C\left(\| f\|_{C^{0,\alpha}(\Omega,\R)}+\|g\|_{C^{1,\alpha}(\partial\Omega,\R)}\right).
\end{equation}
\end{theorem}

\cref{theorem:Schauder} is not yet fully satisfactory in light of assumption \ref{assumption:hoelder-differentiable-densities} as it does not exploit higher regularity of $f = \rho_\mu - \rho_\nu$ in case that $\rho_\mu, \rho_\nu\in C^{k,\alpha}(\Omb,\R) $ for $k\geq 1$.
We therefore use classical Schauder estimates for elliptic PDE to derive higher regularity. Application of the classical Schauder estimate by Agmon, Douglis and Nirenberg  to the situation at hand gives the following.

\begin{theorem}[$C^{k,\alpha}$ Schauder estimate~\cite{agmon1959estimates}]
    \label{theorem:SchauderHigherOrder}
 Let $d>2$ and $k\in \N$ and let $f\in C^{k,\alpha}(\Omb,\R)$. 
 Let furthermore $\Omega$ be bounded and fulfill assumption \ref{assumption:smooth-domain}. 
 If there exists a classical solution to \cref{eq:Poisson}, then actually $u\in C^{2+k,\alpha}(\Omb)$ and there is a constant $C = C_2(\Omega, d, k, \alpha)$ such that the following Schauder estimate holds:
 \begin{equation*}
     \| u\|_{C^{2+k,\alpha}(\Omb)}\leq C \left(\|f\|_{C^{k,\alpha}(\Omb)}+\|g\|_{C^{k+1,\alpha}(\Omb)}\right)
 \end{equation*}
\end{theorem}
\begin{proof}
Within this proof, we use the notation of~\cite{agmon1959estimates}. We apply Theorem 7.3 of this paper in the setting $m=1$, $l_0=2$ and $l=k+2$. All regions $\mathfrak{A}$ and $\mathcal{D}$ are set to $\Omega$ which fulfills the boundary regularity required by the theorem by Assumption A1). We next consider the (principal) symbol $L(\overline{p},p_{d})=|\overline{p}|^2+p_d^2$ in the formal variables $p=(\bar p,p_d)\in\R^{d-1}\times \R$ where we use local coordinates defined by the hemisphere transformations such that the boundary is located in direction $p_d$. As required in \cite[p. 632]{agmon1959estimates}, for $\overline{p}\not=0$, the equation $L(\overline{p},p_{d})=0$ has exactly one root with positive imaginary part in $p_d$, namely $i|\bar p|$. Also condition (1.1) on the same page is trivially fulfilled. The complementing condition on page 633 of this work is on linear independence of the coefficients of diverse boundary condition operators is trivially fulfilled, as we only have one single boundary operator. The smoothness and boundedness of the boundary condition operator is just given by the regularity of $\eta$. Thus the conditions of Theorem 7.3 in our case are fulfilled and we obtain the Schauder estimate for any solution $\tilde u$
\begin{equation*}
    \| \tilde u\|_{C^{2+k,\alpha}(\Omb,\R)}\leq C'(\Omega,d,k,\alpha)\left(\|f\|_{C^{k,\alpha}(\Omb,\R)}+\|g\|_{C^{k+1,\alpha}(\Omb,\R)}+\|\tilde u\|_{C^0(\Omb,\R)}\right). 
\end{equation*}
 Let us now consider the family of solutions $\tilde u\rightarrow \tilde u+c=u$ and choose $c$ such that $\int_\Omega u\, \mathrm{d}x=0$. Now use Theorem \ref{theorem:Schauder} to upper bound the $\|u\|_{C^0(\Omb)}$- term on the right hand side by the right hand side of \eqref{eq:2ndOrderSchauder}. Setting $C_2(\Omega,d,k,\alpha)=C_1(\Omega,d,\alpha)+C'(\Omega,d,k,\alpha)$ and using the obvious inequalities between $C^{k,\alpha}$ H\"older norms for different $k$, we derive the assertion.   
\end{proof}

\subsection{Regular solutions to Beckmann's problem}
The above two theorems can now be applied to our situation with $g = 0$ and $f = \rho_\mu - \rho_\nu$ in order to obtain a Hölder regular solution to the Beckmann problem. We also mention the derived regularity for the vector field $\bm \xi$ and the flow $\bm\Phi_{0,t}(\bm \xi)$. 
The following is the main theorem of this section.
\begin{theorem}[Existence, uniqueness and regularity of solutions of the Beckmann problem]
\label{theorem:Ck_alphaForV}
    Let assumptions \ref{assumption:smooth-domain}, \ref{assumption:uniform-lower-bound} and \ref{assumption:hoelder-differentiable-densities} be fulfilled for $k\in \N_0$. Then,
    \begin{enumerate}[label=(\roman*)]
        \item There exits a potential field $u\in C^{k+2,\alpha}(\overline{\Omega},\R)$  such that $\bm w = \bm \nabla u$ fulfills~\eqref{eq:divergenceConstraint} and $\bm w\in C^{k+1}(\bar \Omega,\R^d)$ with 
        \[
        \|\bm w\|_{C^{k+1,\alpha}(\Omb,\R^d)}\leq C_2(\Omega,d,k,\alpha)\,\|\rho_\mu - \rho_\nu\|_{C^{k,\alpha}(\Omb,\R)}.
        \]
        \item In $W^{1,p}(\Omega,\R^d)$, $p\geq 2$, $\bm w = \bm \nabla u$ is the unique solution to the Beckmann problem  \eqref{eq:min_sq_norm_w} under the constraints \eqref{eq:divergenceConstraint}.
    \end{enumerate}  
\end{theorem}

\begin{proof}
 (i) We  apply Theorem \ref{theorem:Schauder} for $g = 0$ since by assumption A4) and $f = \rho_\mu - \rho_\nu$,  $\int_\Omega f\, \mathrm{d}x=0$ meets the conditions of Theorem \ref{theorem:Schauder}. 
 This gives us $u\in C^{2,\alpha}(\Omb)$. For $k\geq 1$, apply Theorem \ref{theorem:SchauderHigherOrder} in addition.  
 From the $C^{k+2,\alpha}$ regularity of $u$, the $C^{k+1,\alpha}$ regularity of $\bm w$ follows.

 (ii) Clearly, $\bm w\in W^{1,p}(\Omega,\R^d)$ where $\bm w$ is the vector field from (i). 
 Let $\bm w'\in W^{1,p}(\Omega,\R^d)$ fulfill \eqref{eq:divergenceConstraint} and assume $\bm w'\neq \bm w$. 
 Obviously, $j(\tau)= J((1-\tau)\bm w+\tau \bm w')$ is a second order polynomial with non vanishing quadratic coefficient $J(\bm w - \bm w') > 0$.
 Thus, $j(\tau)$ has a global and unique minimum on $\R$ where the first derivative vanishes. 
 By the 1st order optimality condition \eqref{eq:first_order_opt_c} evaluated in the direction $\bm w' - \bm w$, the first derivative vanishes for $\tau = 0$. Thus, $J(\bm w)=j(0)<j(1)=J(\bm w')$.
\end{proof}
In the following, we investigate how \cref{theorem:Ck_alphaForV} can be applied to derive regularity of transport vector fields and their respective flow map.

\subsection{Transport Vector Field and Transport Map Regularity}
Given the solution $\bm w = \bm \nabla u$ from \cref{theorem:Ck_alphaForV} which fulfills $\mathrm{div}\, \bm w = \rho_\mu - \rho_\nu$, we may scale $\bm w$ by the density interpolation
\begin{equation}
\label{eq:interpolateDensities}
\rho_t \coloneq (1 - t) \rho_\nu + t \rho_\mu,~~t\in[0,1], ~~\text{on}~~ \Omega.    
\end{equation}
By assumption \ref{assumption:uniform-lower-bound}, $\rho_t$ is positive with positive lower bound $\kappa$ throughout $\Omega$ such that the vector field
\begin{equation}
    \label{eq:beckmannVectorField}
    \bm \xi(t) \coloneq \frac{\bm w}{\rho(t)}
\end{equation}
is well-defined and $\rho_t$ evolves with $\bm w = \rho(t) \cdot \bm \xi(t)$ according to the transport equation \eqref{eq:continuityEquation} with initial condition $\rho_0=\rho_\nu$.
Together with our previous result, we obtain the following.
\begin{theorem}[Regularity of the transport vector field, flow map and flow end point]
\label{thm:regularity-xi-and-phi}
    Let the assumptions \ref{assumption:smooth-domain}, \ref{assumption:uniform-lower-bound} and \ref{assumption:hoelder-differentiable-densities} be fulfilled for $k \in \N_0$.
    Then,
    \begin{enumerate}[label=(\roman*)]
        \item The vector field $\bm \xi$ defined in~\eqref{eq:beckmannVectorField} lies in $C^{k,\alpha}([0,1]\times \Omb,\R^d)$ with $C^{k,\alpha}$-Hölder norm no larger than
        \begin{equation*}
        \|\bm \xi\|_{C^{k,\alpha}([0,1]\times\Omb,\R^d)}\leq C \left(\kappa^{-1} \max\{\|\rho_\nu\|_{C^{k,\alpha}(\Omb,\R)},\|\rho_\mu\|_{C^{k,\alpha}(\Omb,\R)}\}\right)^{2^k+5}, 
        \end{equation*}
        where $C=C_3(\Omega,d,k,\alpha)$ and $\kappa>0$ is defined in \eqref{eq:kappaPositive}.
        \item If $k\geq 1$, $\bm \xi$ generates a flow $\bm\Phi_{0,t}(\bm \xi)$ with flow endpoint $\bm\Phi(\bm \xi)$ which is a transport map, i.e. $\bm\Phi(\bm \xi)_*\nu=\mu$. 
        Furthermore, $\bm\Phi_{0, t}(\bm \xi),\bm\Phi(\bm \xi)\in C^{k,\alpha}(\Omb,\Omb)$. 
    \end{enumerate}
\end{theorem}
\begin{proof}
(i)  $\rho_t$ is linear affine in $t$ and thus $\rho_t(x)$ is a Hölder function in $C^{k,\alpha}(\overline{\Omega}\times [0,1],\R)$ by assumption \ref{assumption:hoelder-differentiable-densities}. 
Also $\rho_t(x)\geq \kappa$ by assumption \ref{assumption:uniform-lower-bound}. 
We can thus apply (i) and \cite[Proposition A7]{asatryan2023convenient} which provide Hölder continuity for quotients and products to conclude.

 (ii) This follows from (ii) and the fact that for $k\geq 1$, the function $\bm \xi \in C^{k,\alpha}(I \times \Omb,\R)$ is differentiable with bounded first derivative. 
 In particular, it is Lipschitz. 
 By the condition $\bm w\cdot \bm \eta=0$ on $\partial\Omega$, we also obtain $\bm \xi\cdot \bm \eta=0$ at the boundary. 
 Hence the flow $\bm\Phi_{0,t}(\bm \xi)$ applied on some point $x\in\Omega$ never leaves $\Omega$. 
 Along the trajectory, the Lipschitz constant of $\bm \xi$ thus is bounded. 
 We can thus apply standard ODE theory~\cite{HeuserHarro} to prove global existence of a flow $\bm\Phi_{0,t}(\bm \xi)$ for $\bm x\in\Omega$ and $t\in I$.
 That the flow endpoint is a transport map follows from~\eqref{eq:continuityEquation}, \eqref{eq:interpolateDensities} and \eqref{eq:beckmannVectorField}. 
 The $C^{k,\alpha}$-regularity of the flow follows from \cite[Theorem 2.5]{effland2020convergence} for $k=1$. 
 As the authors remark in their proof, their argument can be easily iterated by applying it to equations like $ \frac{\mathrm{d}}{\mathrm{d}t}\left(\nabla \bm\Phi_{s,t}(\bm \xi)\right) = \nabla \bm \xi(t, \bm\Phi_{s,t}(\bm \xi)) \cdot  \left(\nabla \bm\Phi_{s,t}(\bm \xi)\right)$
 and higher order analogues, to derive higher $C^{k,\alpha}$-Hölder regularity for $k\geq 2$.
\end{proof}

\section{Parametric Beckmann Problems
}
\label{sec:parametrized-beckmann}
Modern applications of generative \linebreak learning are frequently concerned with probability distributions and probability densities that depend on external parameters or conditions.
We may envision here language prompts for chat bots or image generation, visual image prompts for applications like inpainting or positional prompts for computer vision models.
Here, we model prompts for conditional generative learning as additional co-variables of the source and target distributions $\rho_\nu$ and $\rho_\mu$ in some parameter set $\Theta \subset \R^q$ with $q \in \N$ and, therefore, of the function $f = \rho_\mu - \rho_\nu$.
In the following sections we investigate the parameter dependence of the target quantities $\bm \xi$ and $\bm\Phi(\bm \xi)$ in the setting where $f$ additionally depends on external parameters $\theta \in \Theta$.

\subsection{Banach-Valued Hölder Functions}
Let $q \in \N$.
We consider $\Theta \subset \R^q$, a bounded, open set with sufficiently smooth boundary (in \cref{sec:Paths}, we will investigate the case of a wider class of time-dependent probability paths where $t \in \Thb = [0, 1]$) such that $\Thb$ is compact and lends itself to notions of of Hölder differentiability.
We define the spaces of abstract functions $f \in C^{l, \beta}(\Thb; C^{k, \alpha}(\Omb))$ that are Hölder differentiable in $\Theta$ with values in a Hölder space as in \cite{lunardi2012analytic}.
That is, for any Banach space $(X, \|\cdot\|_X)$ and a continuous and bounded function $f \in C_{\mathrm{b}}(\Thb;X)$, let 
\begin{equation*}
    [f]_{C^\alpha(\Thb;X)} \coloneq \sup_{\substack{\theta, \vartheta \in \Theta,\\\theta \neq \vartheta}} \frac{\|f(\theta) - f(\vartheta)\|_X}{|\theta - \vartheta|^\alpha}
\end{equation*}
be the Hölder coefficient of $f$.
Then, the space of $\alpha$-Hölder continuous functions $C^\alpha(\Thb;X) \coloneq \{f \in C_{\mathrm{b}}(\Thb;X): [f]_{C^\alpha(\Thb;X)} < \infty\}$ is equipped with the norm 
\[
\|f\|_{C^\alpha(\Thb;X)} \coloneq \|f\|_{C_\mathrm{b}(\Thb;X)} + [f]_{C^\alpha(\Thb;X)},
\] 
where $\|f\|_{C_\mathrm{b}(\Thb;X)} \coloneq \sup_{\theta \in \Theta} \|f(\theta)\|_X$.
Similarly, let the space of functions with bounded Fr\'{e}chet derivatives up to order $k \in \N$ be denoted by 
\begin{equation*}
    \Cb^k(\Thb;X) \coloneq \{f \in C^{\bm m}(\Thb;X): D^{\bm m} f \in \Cb(\Thb;X), \, \forall \bm m \in \N_0^q: |\bm m| \leq k\},
\end{equation*}
where $D^{\bm m}$ denotes the partial (Fr\'{e}chet) derivative operator on functions over $\Thb$ for the multi-index $\bm m$ and equipped with the norm $\|f\|_{\Cb^k(\Thb;X)} \coloneq \max_{|\bm m| \leq k} \|D^{\bm m} f\|_{\Cb(\Thb;X)}$.
Then, the space of $k$ times $\alpha$-Hölder differentiable functions is
\begin{align}
\begin{split}
    C^{k, \alpha}(\Thb;X) \coloneq& \, \{f \in \Cb^k(\Thb;X): D^{\bm m} f \in C^\alpha(\Thb;X), |\bm m| \leq k\}, \\
    \|f\|_{C^{k, \alpha}(\Thb;X)} \coloneq& \, \|f\|_{\Cb^k(\Thb;X)} + \max_{|\bm m| = k}[D^{\bm m} f]_{C^\alpha(\Thb;X)}.
    \label{eq:abstract-hölder}
    \end{split}
\end{align}
With these spaces now set up with $X = C^{k, \alpha}(\Omb)$, we can first conclude respective parameter regularity of the solution $\Theta \ni \theta \mapsto u(\theta) \in C^{2 + k, \alpha}(\Omb)$.

\subsection{Hölder Differentiability of the Beckmann Potential in Time}
\begin{proposition}
    \label{prop:hölder-continuity-u}
    Let $d > 2$, $k \in \N$, $\alpha \in (0, 1)$, $\beta \in (0, 1]$ and $f \in\linebreak  C^{\beta}(\Thb; C^{k, \alpha}(\Omb))$ for $\Omega$ satisfying assumption \ref{assumption:smooth-domain} and $\int_\Omega f(\theta) \dif x = 0$ for all $\theta \in \Theta$.
    Further, for any $\theta \in \Theta$, let $u(\theta) \in C^{2 + k, \alpha}(\Omb)$ denote the (point-wise) mean-value free solution to \cref{eq:Poisson} with right-hand side $f(\theta) \in C^{k, \alpha}(\Omb)$.
    Then, we have $u \in C^\beta(\Thb; C^{2+k, \alpha}(\Omb))$ with $\|u\|_{C^{\beta}(\Thb;C^{2 + k, \alpha}(\Omb))} \leq 2 C_2(\Omega, d, k, \alpha) \|f\|_{C^\beta(\Thb; C^{k, \alpha}(\Omb))}$.
\end{proposition}
\begin{proof}
    First, note that for any $\theta, \vartheta \in \Theta$, the Schauder estimates from \cref{theorem:SchauderHigherOrder} apply to $u(\theta) - u(\vartheta)$ since this difference satisfies the Poisson equation and is also mean-value free.
    Therefore,
    \begin{align*}
        \|u\|_{C^\beta(\Thb; C^{2 +k, \alpha}(\Omb))}
        =& \sup_{\theta, \vartheta \in \Thb} \|u(\theta)\|_{C^{2 + k, \alpha}(\Omb)} + \sup_{\substack{\theta, \vartheta \in \Thb, \\ \theta \neq \vartheta}} \frac{\| u(\theta) - u(\vartheta)\|_{C^{2 + k, \alpha}(\Omb)}}{|\theta - \vartheta|^\beta} \\
        \leq& 2 C \|f\|_{C^\beta(\Thb; C^{k, \alpha}(\Omb))},
    \end{align*}
    where $C = C_2(\Omega, d, k, \alpha)$ is the constant from \cref{theorem:SchauderHigherOrder}.
\end{proof}
The following analogous statement can be derived on $u$ given higher Hölder regularity of $f$.
\begin{corollary}
    \label{cor:hölder-diff-u}
    Let $l \in \N$.
    With the assumptions of \cref{prop:hölder-continuity-u} and $f \in C^{l, \beta}(\Thb; C^{k, \alpha}(\Omb))$.
    Then also $u \in C^{l, \beta}(\Thb; C^{2 + k, \alpha}(\Omb))$ with $\|u\|_{C^{l, \beta}(\Thb; C^{2+k, \alpha}(\Omb))} \leq 2 C_2(\Omega, d, k, \alpha) \|f\|_{C^{l, \beta}(\Thb; C^{k, \alpha}(\Omb))}$.
\end{corollary}
\begin{proof}
    We formalize the idea that for this order of regularity in $\theta$, we can control the parameter-derivatives of $u$ also via the Schauder estimates.
    To this end, consider
    \begin{align*}
        \|u\|_{C^{l, \beta}(\Thb;C^{k, \alpha}(\Omb))} &= \max_{|\bm m| \leq l} \sup_{\theta \in \Theta} \|D^{\bm m} u(\theta)\|_{C^{2 + k, \alpha}(\Omb)}\\
        & + \max_{|\bm m| = l} \sup_{\substack{\theta, \vartheta \in \Theta, \\ \theta \neq \vartheta}} \frac{\|D^{\bm m} u(\theta) - D^{\bm m} u(\vartheta)\|_{C^{2+k, \alpha}(\Omb)}}{|\theta - \vartheta|^\beta}.
    \end{align*}
    We have by linearity that the Fr\'{e}chet derivatives of $u$ of order 1 also satisfies \cref{eq:Poisson} with
    \begin{equation*}
        \Delta [\overline{\nabla}^{\bm m} u(\theta)] = \overline{\nabla}^{\bm{m}} f(\theta)\, \text{ on } \Omega, \quad \bm\nabla [\overline{\nabla}^{\bm m} u(\theta)] \cdot \eta = 0\, \text{ on } \partial \Omega,
    \end{equation*}
    where again $\bm m \in \N_0^q$ with $|\bm m| = 1$.
    By \cref{eq:abstract-hölder}, $\overline{\nabla}^{\bm m} f(\theta) \in C^{k, \alpha}(\Omb)$, so the Schauder estimate in \cref{theorem:SchauderHigherOrder} also holds for $\overline{\nabla}^{\bm m} u(\theta)$.
    Repetition of this argument for derivatives up to order $l$ give analogously to \cref{prop:hölder-continuity-u} $\|u\|_{C^{l, \beta}(\Thb; C^{2+k, \alpha}(\Omb))} \leq 2 C_1 \|f\|_{C^{l, \beta}(\Thb; C^{k, \alpha}(\Omb))}$.
\end{proof}
Therefore, we do indeed obtain the same parameter regularity for $\theta \mapsto u(\theta)$ as we have for $\theta \mapsto f(\theta)$.
The universal approximation results usually applied in statistical learning theory treat the regularity of the parameter variable $\theta$ and the spatial variable $x \in \Omega$ on the same footing, regarding $f$ as a function of $(\theta, x) \in \Theta \times \Omega$ with joint regularity.
One prominent example are time-dependent vector fields (where $\Thb = [0, 1]$) that are learned for NeuralODE and Flow Matching models, see \cite{chen2018neural,lipmanflow}.

Now, on the one hand, $f \in C^{l, \beta}(\Thb; C^{k, \alpha}(\Omb))$ is a rather abstract assumption on $f$ and not really in the spirit of the assumptions made e.g., in \cite[Theorem 2]{belomestny2023simultaneous}.
On the other hand, $f \in C^{k, \alpha}(\Thb \times \Omb)$ is a strictly weaker property than, e.g., $f \in C^{k, \alpha}(\Thb; C^{k, \alpha}(\Omb))$ where
\begin{equation*}
    \|f\|_{C^{k, \alpha}(I \times \Omb)} = \max_{|\bm m| \leq k}\|\partial^{\bm m} f\|_\infty + \max_{|\bm m| = k} \sup_{\substack{(\theta, x), (\vartheta, y) \in \Theta \times \Omega, \\ (\theta,x) \neq (\vartheta, y)}} \frac{|\partial^{\bm m} f(\theta,x) - \partial^{\bm m} f(\vartheta, y)|}{|(\theta,x) - (\vartheta,y)|^\alpha},
\end{equation*}
where $|\bm m| \leq k$ denotes multi-indices $\bm m \in \N_0^{q + d}$ such that $\sum_{i = 1}^{q+d}m_i \leq k$, $\partial^{\bm m}$ denotes the partial derivatives in $\Theta \times \Omega$ and $|(\theta, x)|$ is e.g., the Euclidean norm on $\Theta \times \Omega$.
Note, that $\partial^{\bm m}$ may explicitly denote a mixed derivative in $\theta$ and $x$.
The norm $\|\cdot\|_{\infty}$ is the usual supremum norm over the product $\Theta \times \Omb$.
Since \cref{eq:abstract-hölder} imposes much stricter regularity on $f$, we have the formal inclusion $C^{k, \alpha}(\Thb; C^{k, \alpha}(\Omb)) \subset C^{k, \alpha}(\Thb \times \Omb)$ with
\begin{equation}
    \label{eq:formal-hölder-inclusion}
    \|\cdot\|_{C^{k, \alpha}(\Thb \times \Omb)} \leq \| \cdot\|_{C^{k, \alpha}(\Thb; C^{k, \alpha}(\Omb))}.
\end{equation}
We next turn towards reverse inclusion if we impose higher differentiability on $\Theta \times \Omega$.

\subsection{Inclusion of Banach-Valued Hölder Spaces in Higher-Order Joint Hölder Spaces}
The goal of this section is to find conditions under which $C^{k, \alpha}(\Thb \times \Omb) \subset C^{l, \beta}(\Thb; C^{j, \gamma}(\Omb))$ where we already saw that higher order differentiability on $\Theta \times \Omega$ will be strictly necessary.
Note that classical results on parabolic Hölder spaces (see \cite{lunardi2012analytic}) are of a slightly different nature since such spaces are isotropic and, therefore, a priori do not allow for sufficient control of mixed derivatives.
We therefore investigate $f \in C^{k, \alpha}(\Thb \times \Omb)$ for $k \geq 1$ and first find the following.
\begin{proposition}
    \label{prop:joint-inclusion-in-banach}
    Let $\alpha \in (0, 1)$, $k \in \N$, $k \geq 1$, $\Omega$ be convex and fulfill assumption \ref{assumption:smooth-domain} and $f \in C^{k, \alpha}(\Thb \times \Omb)$.
    For any $\beta$ with $0 < \beta \leq \alpha$, we have $f \in C^\beta(\Thb; C^{k - 1, \alpha}(\Omb))$ with
    \begin{equation*}
        \|f\|_{C^\beta(\Thb;C^{k - 1, \alpha}(\Omb))} \leq C \|f\|_{C^{k, \alpha}(\Thb \times \Omb)}, 
    \end{equation*}
    where $C = C_4(\Omega, d, \alpha) = (2 + \sqrt{d}) (\mathrm{diam}(\Omega))^{1 - \alpha}$.
\end{proposition}
\begin{proof}
    We find estimates to 
    \[
    \|f\|_{C^\beta(\Thb; C^{k-1, \alpha}(\Omb))} = \|f\|_{\Cb(\Thb; C^{k-1, \alpha}(\Omb))} + [f]_{C^\beta(\Thb;C^{k-1, \alpha}(\Omb))}
    \]
    by first noticing that 
    \begin{align*}
        \|f\|&_{\Cb(\Thb; C^{k-1, \alpha}(\Omb))} =
        \sup_{\theta \in \Theta}\|f(\theta)\|_{C^{k-1, \alpha}(\Omb)} \\
        =&
        \sup_{\theta \in \Theta} \Bigg[\max_{|\bm m| \leq k-1} \sup_{x \in \Omega} |\nabla^{\bm m} f(\theta, x)| + \max_{|\bm m| = k-1} \sup_{\substack{x, y \in \Omega,\\x \neq y}} \frac{|\nabla^{\bm m} f(\theta, x) - \nabla^{\bm m} f(\theta, y)|}{|x - y|^\alpha} \Bigg]
    \end{align*}
    which is strictly bounded by $\|f\|_{C^{k, \alpha}(\Thb \times \Omb)}$.
    The Hölder coefficient
    \begin{equation*}
        [f]_{C^\beta(\Thb; C^{k-1, \alpha}(\Omb))}
        = \sup_{\substack{\theta, \vartheta \in \Theta,\\ \theta \neq \vartheta}} \frac{\|f(\theta) - f(\vartheta)\|_{C^{k-1}(\Omb)} + [f(\theta) - f(\vartheta)]_{C^{k-1, \alpha}(\Omb)}}{|\theta - \vartheta|^\beta}
    \end{equation*}
    again requires control of two different terms.
    The term involving the $C^{k-1}$ norm is again bounded by $\|f\|_{C^{k, \alpha}(\Thb \times \Omb)}$ subject to the assumed constraint that $\beta \leq \alpha$, where
    \begin{equation*}
        \sup_{\substack{\theta, \vartheta \in \Theta,\\ \theta \neq \vartheta}} \frac{\|f(\theta) - f(\vartheta)\|_{C^{k-1}(\Omb)}}{|\theta - \vartheta|^\beta} = \sup_{\substack{\theta, \vartheta \in \Theta,\\ \theta \neq \vartheta}} \frac{\max_{|\bm m|\leq k-1} \sup_{x \in \Omega} | \nabla^{\bm m} (f(\theta, x) - f(\vartheta,x))|}{|\theta - \vartheta|^\beta}.
    \end{equation*}
    Up until now, we essentially only used the fact that $f(\theta) \in C^k(\Omb)$.
    However, in order to control the second term, we will actually need the higher regularity of $f \in C^{k, \alpha}(\Thb \times \Omb)$.
    Concretely,
    \begin{align*}
        \sup_{\substack{\theta, \vartheta \in \Theta,\\ \theta \neq \vartheta}}& \frac{[f(\theta) - f(\vartheta)]_{C^{k-1, \alpha}(\Omb)}}{|\theta - \vartheta|^\beta} \\ 
        &= \sup_{\substack{(\theta,x), (\vartheta,y) \in \Theta \times \Omega,\\ (\theta,x) \neq (\vartheta,y)}} \max_{|\bm m| = k-1} \frac{|\nabla^{\bm m} [f(\theta,x) - f(\vartheta,x)] - \nabla^{\bm m} [f(\theta,y) - f(\vartheta,y)]|}{|\theta - \vartheta|^\beta |x - y|^\alpha}.
    \end{align*}
    We rearrange terms in the numerator for fixed $(\theta,x), (\vartheta,y) \in \Theta \times \Omega$ and multi-index $\bm m$ to assemble terms with the same time coordinate and apply the fundamental theorem of calculus to the smooth path $\varphi: \tau \mapsto \tau y + (1 - \tau) x$ twice to obtain
    \begin{align*}
        |\nabla^{\bm m} f(\theta,x) &- \nabla^{\bm m} f(\theta,y) - [\nabla^{\bm m} f(\vartheta,x) - \nabla^{\bm m} f(\vartheta,y)]| \\
        =&\bigg|\int_0^1 \bm \nabla [\nabla^{\bm m} f](\theta, \phi(\tau)) \cdot [y - x] - \bm \nabla[ \nabla^{\bm m} f](\vartheta, \phi(\tau)) \cdot [y - x] \dif \tau \bigg| \\
        \leq& \sup_{\xi \in \Omega} |\bm \nabla [\nabla^{\bm m} f](\theta, \xi) - \bm \nabla [\nabla^{\bm m} f](\vartheta, \xi)| \cdot |y - x|.
    \end{align*}
    Overall, we therefore have with $\mathrm{diam}(\Omega) = \sup_{x, y \in \Omega}|x - y|$
    \begin{align*}
        \sup_{\substack{\theta, \vartheta \in \Theta,\\ \theta \neq \vartheta}}& \frac{[f(\theta) - f(\vartheta)]_{C^{k-1, \alpha}(\Omb)}}{|\theta - \vartheta|^\beta} \\ 
        &\leq \sup_{\substack{(\theta,x), (\vartheta,y) \in \Theta\times \Omega,\\ (\theta,x) \neq (\vartheta,y)}} \max_{|\bm m| = k-1} \frac{\sup_{\xi \in \Omega} |\bm \nabla [\nabla^{\bm m} f](\theta, \xi) - \bm \nabla [\nabla^{\bm m} f](\vartheta, \xi)| \cdot |y - x|}{|\theta - \vartheta|^\beta |x - y|^\alpha} \\
        &\leq \sup_{\substack{\theta, \vartheta \in \Theta, \\ \theta \neq \vartheta}} \max_{|\bm m| = k} \sqrt{d} (\mathrm{diam}(\Omega))^{1 - \alpha} \sup_{\xi \in \Omega} \frac{|\nabla^{\bm m} f(\theta, \xi) - \nabla^{\bm m} f(\vartheta, \xi)|}{|\theta - \vartheta|^\beta} \\
        &\leq \sqrt{d} (\mathrm{diam}(\Omega))^{1 - \alpha} \|f\|_{C^{k, \alpha}(\Thb \times \Omb)}.
    \end{align*}
\end{proof}
The strategy in this proof can be generalized to obtain also Hölder differentiability at the cost of having less regularity in the target Banach space.
\begin{corollary}
\label{cor:joint-inclusion-diff}
    Let $\alpha \in (0, 1)$, $k \in \N$ with $k \geq 1$, $\Omega$ be convex and fulfill assumption \ref{assumption:smooth-domain} and $f \in C^{k, \alpha}(\Thb \times \Omb)$.
    Let $0 < \beta, \gamma \leq \alpha$ and $l, j \in \N_0$ be such that $k - 1 \geq l + j$, then $f \in C^{l, \beta}(\Thb; C^{j, \gamma}(\Omb))$ with 
    \begin{equation*}
        \|f\|_{C^{l, \beta}(\Thb; C^{j, \gamma}(\Omb))} \leq C \|f\|_{C^{k, \alpha}(\Thb \times \Omb)}
    \end{equation*}
    for a constant $C = C_5(\Omega, d, \alpha)$.
\end{corollary}
\begin{proof}
    Again, by the same logic as in \cref{prop:joint-inclusion-in-banach}, we find that 
    \begin{equation*}
        \|f\|_{\Cb^l(\Thb; C^{j, \gamma}(\Omb))}
        = \max_{|\bm m| \leq l} \sup_{\theta \in \Theta} \left[ \|\overline{\nabla}^{\bm m} f(\theta)\|_{C^{j, \gamma}(\Omb)} + [\overline{\nabla}^{\bm m} f(\theta)]_{C^{j, \gamma}(\Omb)} \right]
    \end{equation*}
    contains terms with mixed derivatives up to order $l + j$ which are bounded by\linebreak  $\|f\|_{C^{k, \alpha}(\Thb \times \Omb)}$.
    Next, we regard again $[f]_{C^{l, \beta}(\Thb;C^{j, \gamma}(\Omb))}$, where for $|\bm{m}| \leq l$,
    \begin{align*}
        &\sup_{\substack{\theta, \vartheta \in \Theta,\\\theta \neq \vartheta}} \frac{\|\overline{\nabla}^{\bm{m}} f(\theta) - \overline{\nabla}^{\bm{m}} f(\vartheta)\|_{C^j(\Omb)}}{|\theta - \vartheta|^\beta}
        \\&\hspace{1.5cm} = \sup_{\substack{\theta, \vartheta \in \Theta,\\\theta \neq \vartheta}} \frac{\max_{|\bm n| \leq j} \sup_{x \in \Omega} |\nabla^{\bm n} \overline{\nabla}^{\bm{m}} f(\theta, x) - \nabla^{\bm n} \overline{\nabla}^{\bm{m}} f(\vartheta, x)|}{|\theta - \vartheta|^\beta}
    \end{align*}
    also only involves derivatives of up to order $l + j$ in $\Theta \times \Omega$.
    On the other hand, to account for 
    \begin{align*}
       & \sup_{\substack{\theta, \vartheta \in \Theta,\\\theta \neq \vartheta}} \frac{[\overline{\nabla}^{\bm m} f(\theta) - \overline{\nabla}^{\bm m} f(\vartheta)]_{C^{j, \gamma}(\Omb)}}{|\theta - \vartheta|^\beta}\\
        &= \sup_{\substack{\theta, \vartheta \in \Theta,\\\theta \neq \vartheta}} \frac{1}{|\theta - \vartheta|^\beta} \max_{|\bm n| \leq j} \sup_{\substack{x, y \in \Omega, \\ x \neq y}} \frac{|\nabla^{\bm n} \overline{\nabla}^{\bm m} [f(\theta, x) - f(\vartheta, x)] - \nabla^{\bm n} \overline{\nabla}^{\bm m} [f(\theta, y) - f(\vartheta, y)]|}{|x - y|^\gamma},
    \end{align*}
    we apply the same estimate via the fundamental theorem of calculus to introduce one higher order of derivative, resulting in 
    \begin{equation*}
        \sup_{\substack{\theta, \vartheta \in \Theta,\\\theta \neq \vartheta}} \frac{[\overline{\nabla}^{\bm m} f(\theta) - \overline{\nabla}^{\bm m} f(\vartheta)]_{C^{j, \gamma}(\Omb)}}{|\theta - \vartheta|^\beta} \leq \sqrt{d} (\mathrm{diam}(\Omega))^{1 - \gamma} \|f\|_{C^{k, \alpha}(\Thb \times \Omb)}.
    \end{equation*}
\end{proof}

\subsection{Regular Solutions to the Parametric Beckmann Problem}
We can now piece together a statement similar to \cref{theorem:Ck_alphaForV,thm:regularity-xi-and-phi} at the expense of a certain degree of regularity.
We consider parameter-dependent or conditional probability densities $\rho_\mu = \rho_\mu(\theta, x)$.
\begin{theorem}
    \label{thm:regularity-time-dependent}
    Let $\alpha \in (0, 1)$, $k \in \N$, $k \geq 1$, $\Omega$ be convex and fulfill assumption \ref{assumption:smooth-domain} and $\rho_\nu, \rho_\mu \in C^{k, \alpha}(\Thb \times \Omb)$ fulfill \ref{assumption:uniform-lower-bound} for any $\theta \in \Theta$.
    % We consider $f \coloneq \rho_\nu - \rho_\mu$.
    Then, with $l \coloneq \lfloor k / 2 \rfloor$
    \begin{enumerate}[label=(\roman*)]
        \item There exists a potential field $u \in C^{l, \alpha}(\Thb; C^{l + 1, \alpha}(\Omb))$ such that with $\bm w \coloneq \bm \nabla u$, $\bm w(\theta) \in C^{l, \alpha}(\Omb;\R^d)$ fulfills \cref{eq:divergenceConstraint} for any $\theta \in \Theta$.
        Further, $\bm w \in C^{l, \alpha}(\Thb \times \Omb; \R^d)$ with $\|\bm w\|_{C^{l, \alpha}(\Thb \times \Omb;\R^d)} \leq 2 \sqrt{d} C_2 C_5 \|f\|_{C^{k, \alpha}(\Thb \times \Omb)}$.
        \item The transport vector field $\bm \xi \coloneq \bm w / f$ lies in $C^{l, \alpha}(\Thb \times \Omb; \R^d)$ with 
        \begin{equation}
            \label{eq:parametric-hölder-bound-tvf}
            \|\bm \xi\|_{C^{l, \alpha}(\Thb \times \Omb;\R^d)} \leq C \big[ \tfrac{1}{\kappa} \|f\|_{C^{k, \alpha}(\Thb \times \Omb)} \big]^{2^{l + 5}},
        \end{equation}
        where $C = C_6(\Omega, d, k, \alpha)$ is a constant.
        \item If $l \geq 1$ (i.e., $k \geq 2$), $\bm \xi$ generates a flow $\bm\Phi_{0, t}(\bm \xi)$ with flow end point $\bm\Phi(\bm \xi)$ which is a transport map $\bm\Phi(\bm \xi)_* \nu = \mu$ and $\bm\Phi_{0, t}(\bm \xi), \bm\Phi(\bm \xi) \in C^{l, \alpha}(\Omb; \R^d)$.
    \end{enumerate}
\end{theorem}
\begin{proof}
    This is largely analogous to \cref{theorem:Ck_alphaForV,thm:regularity-xi-and-phi}.
    For (i) note, that with $f = \rho_\mu - \rho_\nu \in C^{k, \alpha}(\Thb \times \Omb)$, \cref{cor:hölder-diff-u,cor:joint-inclusion-diff} we have existence of the potential $u \in C^{l, \alpha}(\Thb; C^{l + 1, \alpha}(\Omb))$ and hence, $\bm \nabla u \in C^{l, \alpha}(\Thb; C^{l, \alpha}(\Omb; \R^d))$ with same regularity in time and the spatial variables.
    \Cref{eq:formal-hölder-inclusion} the also lets us conclude that as a function of both variables, $\bm w \in C^{l, \alpha}(\Thb \times \Omb; \R^d)$.
    Part (ii) follows immediately with $f \geq \kappa$ uniformly and \cite[Proposition A9]{asatryan2023convenient}.
    The existence and regularity of the flow in (iii) is again analogous to \cref{thm:regularity-xi-and-phi}.
\end{proof}
The assumptions of \cref{thm:regularity-time-dependent} are more satisfactory and easier to prove than the abstract condition that $f \in C^{l, \alpha}(\Thb; C^{l, \alpha}(\Omb))$.
We do obtain that the transport vector field $\bm \xi$ which is often the target quantity of universal approximation results fulfills a joint Hölder property on $\Theta \times \Omega$, which is also how results like \cite{belomestny2023simultaneous} are usually applied.
Further, we do obtain the desired spatial regularity of the flow and flow end points in line with the regularity of the Beckmann potential $u$.
Towards controlling the model error in empirical risk minimization, we can apply the approximation result due to Belomestny et al.\ (see \cref{sec:appendix-belomestny}) which holds for deep neural networks with $\mathtt{ReQU}$ activation functions.
We immediately find the following consequence.
\begin{proposition}{($C^l$-Approximation of Transport Vector Fields)}
    Let $\alpha \in (0, 1)$, $k \in \N$ with $k \geq 1$, $\Theta = (0, 1)^q$, $\Omega$ compactly contained in $[0, 1]^d$ fulfilling assumption \ref{assumption:smooth-domain} and $\rho_\nu, \rho_\mu \in C^{k, \alpha}(\overline{\Theta} \times \overline{\Omega})$ fulfill assumption \ref{assumption:uniform-lower-bound} for any $\theta \in \Theta$.
    Let $\bm\xi$ be the transport vector field from \cref{thm:regularity-time-dependent} with $H$ the right-hand side of \cref{eq:parametric-hölder-bound-tvf}.
    Then, for any integer $K \geq 2$, there exists a $\mathtt{ReQU}$ neural network $\bm\xi_\theta:[0, 1]^{q + d} \to \R^d$ of width $\max\{4 (q+d) (K + l)^{q+d}, 12 ((K + 2l) + 1\}$, number of hidden layers
    \begin{equation*}
        6 + 2 (l - 2) + \lceil \log_2 (q + d)\rceil + 2 \max \{ \lceil \max\{\log_2( (2 l + 1) (q + d)), \log_2 \log_2 H\rceil, 1\}
    \end{equation*}
    and at most $(q + d) (K + l)^{q+d} C(l, q+d, H)$ non-zero weights taking their values in $[-1, 1]$ such that for any $m \in \{0, \ldots, l\}$,
    \begin{equation*}
        \|\bm\xi - \bm\xi_\theta\|_{C^m(\Thb \times \Omb; \R^d)} \leq 
        \frac{\widetilde{C}(q+d, l, \alpha, H)}{K^{l + \alpha - \ell}}.
    \end{equation*}
\end{proposition}
\begin{proof}
    First, we extend $\bm\xi \in C^{l, \alpha}(\Thb\times\Omb; \R^d)$ to $\widehat{\bm\xi} \in C^{l, \alpha}([0, 1]^{q + d}; \R^d)$ such that under restriction, $\widehat{\bm\xi}|_{\Thb\times\Omb} \equiv \bm\xi$.
    Fixing $K \geq 2$, \cref{thm:belomestny} guarantees the existence of a deep $\mathtt{ReQU}$ neural network $\bm\xi_\theta$ with the desired architecture bounds and $\|\bm\xi - \bm\xi_\theta|_{\Thb\times\Omb}\|_{C^m(\Thb\times\Omb;\R^d)} \leq \|\widehat{\bm\xi} - \bm\xi_\theta\|_{C^{m}([0, 1]^{q+d};\R^d)}$ holds.
\end{proof}

\section{Regular Transport Vector Fields for Generic Probability Paths}
\label{sec:Paths}
Our results from \cref{sec:Beckmann} are based on a particular construction of the probability evolution from $\rho_\nu$ to $\rho_\mu$, namely linear interpolation via \cref{eq:interpolateDensities}.
However, the transport from $\mu$ to $\nu$ (resp.\ $\rho_\mu$ to $\rho_\nu$) along a variety of other \emph{probability paths} $\rho: [0, 1] \to X$ has been investigated in the past, where $X$ is some Banach space of suitably regular probability density functions over the common domain $\Omega$.
In the present setting, we will again be interested in the case that for $t \in (0, 1)$, $\rho(t)$ is a probability density function in $X = C^{k, \alpha}(\Omb)$.
Next, we ask whether the construction introduced in \cref{sec:Beckmann} can also be applied in a similar manner to derive existence and regularity of a transport vector field and flow given that the prescribe that the induced probability density transport follows some probability path $\rho$.

\subsection{Generalized Setting for Probability Paths}
The natural generalization of \cref{eq:min_sq_norm_w,eq:divergenceConstraint} is that $\bm w$ is already a time-dependent vector field which fulfills for all $t \in I := (0, 1)$
\begin{equation}
    \label{eq:time-dependent-divergence}
    \bm\nabla\cdot \bm w(t) = \dot{\rho}(t), \quad \text{ on } \Omega, \qquad \bm w(t) \cdot \bm \eta = 0 \quad \text{ on } \partial\Omega
\end{equation}
and we introduce the convex minimization objective $J_{I \times \Omega}(\bm w) = \linebreak \frac{1}{2} \int_I \int_\Omega |\bm w(t, x)|^2 \dif x \, \dif t$.
The dot in \cref{eq:time-dependent-divergence} denotes the time derivative of $t \mapsto \rho(t)$.
The instantaneous rate of change equation \cref{eq:time-dependent-divergence} already requires a suitable degree of regularity of $\rho$ to be well-defined.
Reformulation into an unconstrained optimization problem and derivation of the first order optimality conditions is completely analogous to \cref{sec:Unconstrained} and yields in particular that $J_{I\times\Omega}$ can be minimized point-wise for each fixed $t$ resulting in
\begin{equation*}
    \bm\nabla\cdot \bm w(t) = f(t) \text{ and } \bm w(t) = \bm \nabla u(t) \quad \text{ a.e. on }\Omega,
\end{equation*}
where we have introduced $f = \dot{\rho}$.
Altogether, we obtain a parameter-dependent Neumann problem for $u$ of the kind investigated also in \cref{sec:parametrized-beckmann}.
In order to apply \cref{thm:regularity-time-dependent}, $\dot{\rho}$ must also fulfill the Fredholm alternative that is, $\int_\Omega \dot{\rho}(t) \dif x = 0$ for all $t \in I$.
Given a sufficiently regular solution $u$ of the Neumann problem, we may then construct the vector field $\bm \xi(t) = \bm w(t) / \rho(t) = [\bm \nabla u(t)] / \rho(t)$ and its respective flow $\bm\Phi_{0, t}(\bm \xi)$ and flow end point $\bm\Phi(\bm \xi)$.
For this construction to be well-defined, we introduce the following assumption
\begin{assumptions}
    \label{alt-assumptions-dummy}
    For a probability path $\rho$ as a function of several variables $\rho: (t, x) \mapsto \rho(t, x) \in \R$ (respectively, as a Banach-valued function $\rho: t \mapsto \rho(t) \in X$), we assume the following alternative to assumption \ref{assumption:uniform-lower-bound}:
    \begin{enumerate}[label=B2)]
        \item For all $(t, x) \in I \times \Omega$, we have that $\rho(t, x) > 0$ (respectively, $\rho(t, x) \coloneq [\rho(t)](x) > 0$) and
        \label{assumption:alternative-lower-bound}
        \begin{equation*}
            \kappa \coloneq  \inf_{t \in I} \inf_{x \in \Omega} \rho(t, x) >0.
        \end{equation*}
    \end{enumerate}
\end{assumptions}
Now, transporting $\rho(0) = \rho_\nu$ along the flow $\bm\Phi_{0, t}(\bm \xi)$ defines a probability path $\widetilde{\rho}$ with $\widetilde{\rho}(0) = \rho_\nu$.
It is well-known that $\bm \xi$ and $\widetilde{\rho}$ satisfy the transport equation
\begin{equation}
    \label{eq:continuity-equivalence-rhos}
    \frac{\partial}{\partial t} \widetilde{\rho}(t) = - \bm\nabla\cdot (\rho(t) \, \bm \xi(t))
    = - \Delta \,\, u(t) = \dot{\rho}(t), \quad \text{ a.e. on } \Omega.
\end{equation}
Since $\rho$ and $\widetilde{\rho}$ coincide at $t = 0$, we also know that the flow $\bm\Phi_{0, t}(\bm \xi)$ defines a probability path entirely along $\rho$ and we have control of the regularity of $\bm \xi$ and $\bm\Phi(\bm \xi)$ as desired.
\begin{theorem}
    \label{thm:existence-and-regularity-generic-paths}
    Let $\alpha \in (0, 1)$, $l, j \in \N$, $l \geq 1$, $\Omega$ be convex and fulfill assumption \ref{assumption:smooth-domain} and $\rho \in C^{l, \alpha}(\overline I; C^{j, \alpha}(\Omb))$ satisfy assumption \ref{assumption:alternative-lower-bound} and $\int_{\Omega} \dot{\rho}(t) \dif x = 0$ for all $t \in I$.
    Then,
    \begin{enumerate}[label=(\roman*)]
        \item There exists a potential field $u \in C^{l-1, \alpha}(\overline I;C^{j+2, \alpha}(\Omb))$ such that with $\bm w \coloneq \bm \nabla u$, $\bm w(t) \in C^{j + 1, \alpha}(\Omb; \R^d)$ fulfills \cref{eq:time-dependent-divergence} for each $t \in I$.
        \label{item:existence-u-for-paths}
        \item The transport vector field $\bm \xi = \bm w / \rho$ lies in $C^{l - 1, \alpha}(\overline I; C^{j, \alpha}(\Omb))$ and for $k \coloneq \min\{l-1, j\}$ also in $C^{k, \alpha}(\overline I \times \Omb; \R^d)$ with
        \begin{equation*}
            \|\bm \xi\|_{C^{k, \alpha}(\overline I \times \Omb; \R^d)} \leq 
            C \left( \kappa^{-1} \|\rho\|_{C^{l, \alpha}(\overline I; C^{j, \alpha}(\Omb))} \right)^{2^{k + 5}},
        \end{equation*}
        where $C = C_7(\Omega, d, k, \alpha)$.
        \label{item:regularity-vf-for-paths}
        \item For $l, j \geq 1$, $\xi$ generates a flow $\bm\Phi_{0, t}(\bm \xi) \in C^{j, \alpha}(\Omb; \Omb)$ for all $t \in I$ that generates the probability path $\rho$.
        \label{item:regularit-flow-for-paths}
    \end{enumerate}
\end{theorem}
\begin{proof}
    Item \ref{item:existence-u-for-paths} follows immediately with $\dot{\rho} \in C^{l - 1, \alpha}(\overline I; C^{j, \alpha}(\Omb))$ and \linebreak \cref{cor:hölder-diff-u}.
    For item \ref{item:regularity-vf-for-paths}, note that $\bm w \in C^{k, \alpha}(\overline I \times \Omb; \R^d)$ and $\rho \in C^{k, \alpha}(\overline I \times \Omb)$, such that with \cite[Proposition A9]{asatryan2023convenient} and assumption \ref{assumption:alternative-lower-bound}, we immediately have the desired bound with
    \begin{equation*}
        C_7(\Omega, d, k, \alpha) = (2 \max\{2 \sqrt{d} C_2 C_5, 1\})^{2^{k + 5}} \cdot \max\{1, \mathrm{diam}(I \times \Omega)^{2(1 - \alpha)}\}
    \end{equation*}
    by \cref{thm:regularity-time-dependent}.
    In order to verify item \ref{item:regularit-flow-for-paths}, first note that $\bm \xi$ is at least continuous in $t$ and in particular Lipschitz in $x$ such that the flow $\bm\Phi_{0, t}(\bm \xi) \in C^{j, \alpha}(\Omb; \Omb)$ exists analogously to before.
    Since both $\rho$ and the probability path $\widetilde{\rho}$ defined by $t \mapsto \bm\Phi_{0, t}(\bm \xi) _* \nu$ coincide at time $t = 0$ and evolve according to the transport equation \cref{eq:continuity-equivalence-rhos} in the strong sense, we have $\rho(t) = \widetilde{\rho}(t)$ for all $t \in I$.
\end{proof}

\begin{remark}
    Arguably, the nicer property to require from $\rho$ is joint Hölder continuity $\rho \in C^{k, \alpha}(\overline I \times \Omb)$.
    Then, by \cref{cor:joint-inclusion-diff}, with $k - 1 \geq l + j$, we still have $\rho \in C^{l, \alpha}(\overline I; C^{j, \alpha}(\Omb))$ and requiring $l \geq 1$, we can proceed analogously.
    If $k \geq 3$, again by $l \coloneq \lfloor k / 2 \rfloor$, we have $\bm \xi \in C^{l - 1, \alpha}(\overline I \times \Omb; \R^d)$ as a potential target for universal approximation.
    Since $\bm\Phi_{0, t}(\bm \xi)$ exists as soon as $\bm \xi$ is continuous in $t$, it is advantageous to instead set $l = 1$ and from $\bm \xi \in C^{0, \beta}(\overline I; C^{k - 2, \alpha}(\Omb))$ derive that $\bm\Phi_{0, t}(\bm \xi) \in C^{k - 2, \alpha}(\Omb;\Omb)$ for all $t \in I$.
    Note also, that analogously to \cref{sec:parametrized-beckmann}, one can introduce a parameter dependence in the probability path $\rho$.
    Requiring sufficient conditions for joint Hölder differentiability on $\overline{I}\times\Thb\times\Omb$ again allows for a jointly Hölder differentiable transport vector field along the parametric path $\rho = \rho(t, \theta, x)$.
\end{remark}

\subsection{Example: The KL Fisher-Rao Gradient Flow}
Training of generative models can be viewed also through the eyes of gradient flows~\cite{ambrosio2005gradient,ambrosio2007gradient} in which some differentiable divergence $\mathsf{d}$ between a target measure $\mu$ and a model $\widehat{\mu}$ is minimized with respect to $\widehat{\mu}$.
That is, for $\widehat{\mu} \in \mathcal{H}$, some sufficiently large space $\mathcal{H}$ like the Wasserstein space, the functional $\widehat{\mu} \mapsto \mathsf{d}(\mu\| \widehat{\mu})$ is regarded as a minimization objective or energy functional.
Following the negative (functional) gradient for the Kullback-Leibler divergence
\begin{equation*}
    - \frac{\delta}{\delta \widehat{\mu}} \mathsf{d}_{\mathrm{KL}}(\mu\|\widehat{\mu})
    = - \frac{\delta}{\delta \widehat{\mu}} \int \log \left( \frac{\mathrm{d}\mu}{\mathrm{d} \widehat{\mu}} \right) \,\mathrm{d}\mu
\end{equation*}
gives rise to a gradient descent system typically called the ``gradient flow''.
For probability measures on a compact domain with positive densities, as considered here, this gradient flow is well-defined and a solution for the density flow can be given in closed form by \cite{carrillo2024fisher,chen2026sampling}
\begin{equation*}
    \label{eq:fisher-rao-gradient-flow}
    \rho_t(x) = \frac{1}{Z_t} \rho_\nu^{1 - t}(x) \cdot \rho_\mu^t(x)\qquad \forall t \in [0,1],
\end{equation*}
where $Z_t = \int_\Omega [\rho_\nu^{1 - t}(x) \cdot \rho_\mu^t(x)] \, \mathrm{d}x$ is a time-dependent normalization constant.

Considering the strict positivity of densities we assume here, this form may be viewed as a linear interpolation of the Gibbs potentials $V_\mu = - \log(\rho_\mu)$ and $V_\nu = - \log(\rho_\nu)$ of $\rho_\mu = \mathrm{e}^{- V_\mu}$ and $\rho_\nu = \mathrm{e}^{- V_\nu}$, respectively.
That is, we have
\begin{equation*}
    \rho_t(x) = \frac{1}{Z_t} \cdot \mathrm{e}^{- V_t(x)}, \quad Z_t = \int_{\Omega} \mathrm{e}^{- V_t(x)} \, \mathrm{d}x
\end{equation*}
where $V_t = (1 - t) V_\nu + t V_\mu$.
Now, under suitable smoothness assumptions on the domain $\Omega$, the following corollary holds. 
\begin{corollary}
    Let $\alpha \in (0,1)$, $k \in \N$, $\Omega$ be convex and fulfill assumption \ref{assumption:smooth-domain}, $\rho_\nu, \rho_\mu \in C^{k, \alpha}(\Omb)$ fulfill assumptions \ref{assumption:uniform-lower-bound} and \ref{assumption:hoelder-differentiable-densities} and $\rho$ the Fisher-Rao gradient flow \eqref{eq:fisher-rao-gradient-flow}.
    Then, there is a transport vector field $\bm \xi \in C^{k, \alpha}(\overline I \times \Omb;\R^d)$ with flow $\bm\Phi_{0, t}(\bm \xi) \in C^{k, \alpha}(\Omb; \R^d)$ for all $t \in I$ that generates $\rho$.
\end{corollary}
\begin{proof}
    We first note that the Fisher-Rao gradient flow is smooth in time.
    To see this, note that $\kappa \leq \rho_\nu, \rho_\mu \leq K$, so \cite[Prop.~A.5, A.10]{asatryan2023convenient} guarantees for each $t \in I$ we have that $V_t \in C^{k, \alpha}(\Omb)$.
    In particular, assumption \ref{assumption:alternative-lower-bound} is fulfilled.
    Now, we show that the parameter integral $t \mapsto Z_t$ is differentiable by passing the derivative under the integral.
    To this end, $\rho_\mu^t \in L^{1/t}(\Omega, \dif x)$ and $\rho_\nu^{1-t} \in L^{1/(1-t)}(\Omega, \dif x)$ for all $t \in I$.
    Therefore, $\mathrm{e}^{- V_t}$ is majorized $t$-independently with $\|\mathrm{e}^{-V_t}\|_{L^1} \leq \|\rho_\mu^t\|_{L^{1/t}} \cdot \|\rho_\nu^{1-t}\|_{L^{1/(1-t)}} = 1$ due to Hölder's inequality.
    Further, by straight-forward computation, $\dot{\rho}_t = \rho_t \cdot \log(\rho_\mu / \rho_\nu)$, showing that $t \mapsto \mathrm{e}^{- V_t(x)}$ is $C^1$ for fixed $x \in \Omega$ and that $\dot{\rho}_t \in C^{k, \alpha}(\Omb)$ for fixed $t \in I$.
    Hence, $\dot{\rho} \in C^{k, \alpha}(\overline I; C^{k, \alpha}(\Omb))$.
    Finally, continuity on the bounded domain $\Omega$ gives a uniform upper bound $K$ allowing for 
    \begin{equation*}
        |\dot{\rho}_t(x)| = \left| \rho_t(x) \cdot \log\left( \frac{\rho_\mu(x)}{\rho_\nu(x)}\right) \right|
        \leq K \cdot \log \frac{K}{\kappa}.
    \end{equation*}
    In order to show that $f_t \coloneq \dot{\rho}_t$ also satisfies the Fredholm alternative, we require $\int_\Omega \dot{\rho}_t(x) \dif x = 0$ for all $t \in [0,1]$.
    This is immediate as soon as we can again interchange integration and differentiation.
    Again, we derive an $L^1$ majorizing constant by
    \begin{equation*}
        \|\dot{\rho}_t\|_{L^1} \leq \frac{1}{|Z_t|} \log\left( \frac{K}{\kappa}\right) \|\rho_\mu^t \rho_\nu^{1-t}\|_{L^1}
        \leq \frac{1}{\kappa \dif x(\Omb)} \cdot \log \left( \frac{K}{\kappa}\right).
    \end{equation*}
    The map $t \mapsto \dot{\rho}_t(x)$ is $C^1$ for $x \in \Omega$ and $t \in I$ since $\ddot{\rho}_t(x) = \dot{\rho}_t(x) \cdot \log(\rho_\mu(x)/\rho_\nu(x))$ is clearly continuous in $t$.
    Finally, $\ddot{\rho}_t$ has a uniform upper bound
    \begin{equation*}
        |\ddot{\rho}_t(x)| = \left| \dot{\rho}_t(x) \cdot \log \left( \frac{\rho_\mu(x)}{\rho_\nu(x)} \right)\right|
        \leq K \cdot \left(\log \left( \frac{K}{\kappa}\right)\right)^2.
    \end{equation*}
    
    Therefore, $f_t \in C^{k, \alpha}(\Omb)$ satisfies the Fredholm alternative and $\rho_t(x) \geq \linebreak \kappa (K \dif x(\Omb))^{-1}$.
    The statement now follows from \cref{thm:existence-and-regularity-generic-paths}.
\end{proof}

\section{Conclusion} \label{sec:Conclusion}
In this article, we determined the regularity of solutions to Beckmann's problem for the case of Hölder continuous, positive densities on a compact regular domain using Schauder estimates. Given that Schauder estimates are not only available for Hölder spaces, but also for Sobolev spaces \cite{agmon1959estimates}, a natural generalization of our findings would be to extend these to the Sobolev case. 

We also study the joint Hölder regularity in the data and parameter dimension of parametric optimal transport vector fields via the Beckmann approach. Thereby we provide a sound foundation for universal approximation for conditional generative learning.

In generative learning, often Gaussian noise is utilized as the source measure. But Gaussian noise has unbounded support and therefore does not fall into the class of problems considered here. A study extending our approach from bounded to unbounded domains would therefore be useful. Note that division of the flux field by densities from the probability path to obtain the  flow vector fields will yield diverging vector fields when approaching infinity and potentially these are only locally Lipschitz. While this is enough to establish the existence of local flows \cite{HeuserHarro}, new techniques to establish the global existence of flow maps for the entire time interval $t\in [0,1]$ are required. To achieve this, the absence of runaway solutions has to be established. While this is highly likely to be true, due to mass conservation in the probability path, the technical details of a proof are not obvious. 

Last but not least, Beckmann flows are an interesting option to design  new generative learning algorithms, which should be tested empirically.      

\section*{Acknowledgments}
Interesting discussions with E. Erhardt,  S.\ Neumayer, G.\ Steidl and S.\
Wang are gratefully acknowledged.

\bibliographystyle{siamplain}
\bibliography{references}

\appendix
\section{Universal Approximation with $\mathtt{ReQU}$ Neural Networks}
\label{sec:appendix-belomestny}
The main part of this work is concerned with the existence of a transport vector field $\bm\xi$ such that its flow $\bm\Phi(\bm\xi)$ is a diffeomorphic transport map $\bm\Phi(\bm\xi)_*\nu = \mu$ for $\mu$ and $\nu$ admitting $C^{k, \alpha}$-Hölder differentiable Lebesgue densities $\rho_\mu, \rho_\nu$.
For completeness, we give some additional background on the universal approximation background of our derived regularity results.

That a $C^{k, \alpha}$-Hölder differentiable transport vector field from $\rho_\nu$ to $\rho_\mu$ exists is a central assumption when applying approximation results~\cite{belomestny2023simultaneous} to prove PAC learning for empirical risk minimization schemes and bounding the resulting model error.
Control of the model error often boils down to controlling $\|\bm\xi - \bm\xi_\theta\|_C^{\ell}(\Omb)$ for some $\ell \in \N_0$, where $\bm\xi$ is the true vector field transporting $\nu$ to $\mu$ and $\bm\xi_\theta$ is from some hypothesis class $\mathcal{H}^{\bm\xi}$ of model vector fields, e.g., parametrized by deep neural networks.

The work by Belomestny et al.\ gives a constructive approximation of $C^{k, \alpha}$ vector fields with deep $\mathtt{ReQU}$ neural networks using B-splines.
Deep neural network architectures are considered fully connected with the $\mathtt{ReQU}$ activation $\sigma(x) = (\max\{0, x\})^2$.
Ultimately, the architecture constructs tensor spline bases with number of knots $K$ along each dimension which implicitly becomes an architecture hyperparameter as it can be freely chosen.

\begin{theorem}{\cite[Thm.~2]{belomestny2023simultaneous}}
    \label{thm:belomestny}
    Let $k \in \N$, $k \geq 2$, $\alpha \in (0, 1)$ and $p, d \in \N$.
    Then, for any $H > 0$, for any $f: [0, 1]^d \to \R^p$ such that $f \in C^{k, \alpha}([0, 1]^d;\R^p)$ and $\|f\|_{C^{k, \alpha}([0, 1]^d;\R^p)} \leq H$ and any integer $K \geq 2$, there exists a neural network $h_f:[0, 1]^d \to \R^p$ of width $\max\{(4d (K + k)^d, 12 ((K + 2k) + 1), p\}$, number of hidden layers
    \begin{equation*}
        6 + 2(k - 2) + \lceil \log_2d\rceil + 2 \max\big\{ \lceil\max\{\log_2(2dk + d), \log_2\log_2 H\rceil, 1\big\}
    \end{equation*}
    and at most $p (K + k)^d \cdot C(k, d, H)$ non-zero weights taking their values in $[-1, 1]$ such that for any $\ell \in \{0, \ldots, k\}$,
    \begin{equation*}
        \|f - h_f\|_{C^\ell([0, 1]^d;\R^p)} \leq (9^{d(k-1)} (2 k + 1)^{2d + \ell} + 1) \frac{(\sqrt{2} e d)^{k+\alpha} H}{K^{k + \alpha - \ell}}.
    \end{equation*}
\end{theorem}
The explicit higher regularity error bound given in this result are highly useful for flow-based generative modeling, as the Jacobian determinant of the flow map enters the transported density \eqref{eq:transformation_of_densities}, so $C^1$ control is vital.
The explicit bounds on the depth and the width can be used to adaptively increase the hypothesis space $\mathcal{H}^{\bm\xi}$ with the sample size $n \in \N$ to derive explicit sample complexity bounds for PAC learning.
\end{document}

%% file: ex_shared.tex
\usepackage{lipsum}
\usepackage[T1]{fontenc}
\usepackage{amsfonts}
\usepackage{graphicx}
\usepackage{epstopdf}
\usepackage{enumitem}
\usepackage{bm}
\usepackage{algorithmic}
\ifpdf
  \DeclareGraphicsExtensions{.eps,.pdf,.png,.jpg}
\else
  \DeclareGraphicsExtensions{.eps}
\fi
\usepackage{mathtools}

\newsiamremark{remark}{Remark}
\newsiamremark{hypothesis}{Hypothesis}
\crefname{hypothesis}{Hypothesis}{Hypotheses}
\newsiamthm{claim}{Claim}
\newsiamremark{fact}{Fact}
\crefname{fact}{Fact}{Facts}

\headers{Regularity for Beckmann Problems}{H.\ Gottschalk, and T.\ Riedlinger}

\title{%Hölder 
Regularity of Solutions to  Beckmann's Parametric Optimal Transport
}

\author{Hanno Gottschalk\thanks{Institute of Mathematics, TU Berlin (\email{\{gottschalk,riedlinger\}@math.tu-berlin.de}).}
\and Tobias J.\ Riedlinger\footnotemark[1]}

\usepackage{amsopn}

\newtheorem{assumptions}[theorem]{Assumptions}
\newcommand{\N}{\mathbb{N}}
\newcommand{\R}{\mathbb{R}}
\newcommand{\Omb}{\overline{\Omega}}
\newcommand{\Thb}{\overline{\Theta}}
\newcommand{\Cb}{C_{\mathrm{b}}}

\newcommand{\dif}{\,\mathrm{d}}